\documentclass[11pt]{article}

\usepackage{mathtools, mathrsfs, amssymb, stmaryrd}
\usepackage{graphicx, xcolor, tikz-cd}

\usepackage{mathpazo}

\usepackage{enumitem}
\setlist[enumerate]{noitemsep}
\setlist[enumerate, 1]{label = (\alph*)}
\setlist[itemize]{noitemsep}

\usepackage[left=1in, right=1in]{geometry}

\usepackage{amsthm}

\usepackage{hyperref}
\hypersetup{
    colorlinks,
    linkcolor={red!50!black},
    citecolor={blue!75!black},
    urlcolor={blue!90!black}
}

\usepackage[capitalise, nameinlink, noabbrev]{cleveref}

\crefname{equation}{}{} 
\numberwithin{equation}{subsection}

\newtheorem{lemma}[equation]{Lemma}
\newtheorem{proposition}[equation]{Proposition}
\newtheorem{corollary}[equation]{Corollary}
\newtheorem*{corollary*}{Corollary}

\newlist{propositionenum}{enumerate}{1}
\setlist[propositionenum,1]{
    label=(\alph*),
    ref={\theproposition(\alph*)},
    noitemsep,
}
\crefname{propositionenumi}{Proposition}{Propositions}

\newtheorem{maintheorem}{Theorem}
\renewcommand{\themaintheorem}{\arabic{maintheorem}} 
\crefname{maintheorem}{Theorem}{Theorems}

\newlist{maintheoremenum}{enumerate}{1}
\setlist[maintheoremenum,1]{
    label=(\alph*),
    ref={\themaintheorem(\alph*)},
}
\crefname{maintheoremenumi}{Theorem}{Theorems}

\theoremstyle{definition}
\newtheorem{definition}[equation]{Definition}
\newtheorem{remark}[equation]{Remark}
\newtheorem{example}[equation]{Example}
\newtheorem{notation}[equation]{Notation}
\newtheorem*{notation*}{Notation}

\newlist{definitionenum}{enumerate}{1}
\setlist[definitionenum,1]{
    label=(\arabic*),
    ref={\thedefinition(\arabic*)},
    noitemsep,
}
\crefname{definitionenumi}{}{}

\newcommand{\deftextcommand}[1]{%
  \expandafter\providecommand\csname #1\endcsname{\mathrm{#1}}%
}
\newcommand{\deftextcommands}[1]{%
    \forcsvlist{\deftextcommand}{#1}%
}

\deftextcommands{id,ker,coker,im,
res,tr,nm,Ind,Res,Coind,
Hom,End,Aut,Ext,Tor,Pic,Gal,
GL,SL,PGL,PSL,AGL,
Fun,Ho,Mod,Perf,Fin,N,sSet,Cat,Wald,Sp,
rep,Rep,stmod,StMod,vect,Vect,proj,Proj,
K,G,HH,HC,THH,TC,TP,Wh,
op,sg}

\newcommand{\bbZ}{\mathbb{Z}}

\newcommand{\bbF}{\mathbb{F}} 
 
\newcommand{\bbT}{\mathbb{T}}

\newcommand{\cA}{\mathcal{A}}
\newcommand{\cB}{\mathcal{B}}
\newcommand{\cC}{\mathcal{C}}

\newcommand{\cE}{\mathcal{E}}
\newcommand{\cF}{\mathcal{F}}
\newcommand{\cO}{\mathcal{O}}
\newcommand{\cP}{\mathcal{P}}

\newcommand{\cR}{\mathcal{R}}
\newcommand{\cS}{\mathcal{S}}

\newcommand{\cW}{\mathcal{W}}
\newcommand{\cX}{\mathcal{X}}

\DeclareMathOperator*{\colim}{colim}
\DeclareMathOperator*{\hocolim}{hocolim}

\newcommand{\wed}{\vee}

\newcommand{\sus}{\Sigma}

\newcommand{\WhTC}{\Wh^{\TC}}
\newcommand{\Ksg}{\K^\sg}

\newcommand{\ul}{\underline}

\newcommand{\wt}{\widetilde}

\newcommand{\tEG}{\wt{EG}}

\newcommand{\nv}{^{-1}}

\renewcommand\epsilon\varepsilon

\renewcommand\emptyset\varnothing

\title{Derived induction theory for the K-theory of modular group algebras}
\author{Chase Vogeli}
\date{October 29, 2025}

\begin{document}

\newpage
\maketitle

\begin{abstract}
    We prove an induction theorem for the higher algebraic K-groups of group algebras $kG$ of finite groups $G$ over characteristic $p$ finite fields $k$. For a certain class of finite groups, which we call $p$-isolated, this reduces calculations to calculations for their $p$-subgroups. We do so by showing that the stable module categories of $kH$ as $H$ ranges over subgroups of $G$ assemble into a categorical Green functor, which results in a spectral Green functor structure on K-theory. By general induction theory, this reduces proving a spectrum-level induction statement to proving an induction statement on $\pi_0$ Green functors, which we accomplish using modular representation theory. For $p$-isolated groups with Sylow $p$-subgroups of order $p$, we produce explicit new calculations of K-groups.
\end{abstract}
    
\setcounter{tocdepth}{2}
\tableofcontents

\newpage
\section{Introduction and main results}

Algebraic K-theory is a fundamental invariant of rings that connects to mathematics ranging from geometric topology to number theory. When applied to group algebras, K-theory houses important invariants such as Wall's finiteness obstruction or Whitehead torsion. 

While the low-degree K-groups of integral group algebras $\bbZ G$ have been extensively studied \cite{Oliver:Whitehead}, complete calculations in high degrees are generally intractable. Indeed, this is already the case for trivial $G$: Kurihara showed that determining the remaining unknown K-groups of $\bbZ$ would resolve the Kummer--Vandiver on class groups of cyclotomic fields \cite{Kurihara}. 

In contrast, complete knowledge of K-groups is possible over finite fields. Such computations began with Quillen's seminal 1972 work, which computed the K-groups of finite fields themselves \cite{Quillen:finite-field}. Further calculations relied on the the development of topological cyclic homology (TC) in the following decades, which has computational implications for K-theory via the cyclotomic trace. Using these, Hesselholt--Madsen computed the K-groups of truncated polynomial algebras \cite{Hesselholt--Madsen:truncated}, which includes group algebras of finite cyclic $p$-groups \cite{Madsen:survey}. Recent work of Lück--Reich--Rognes--Varisco establishes an induction result for the TC of group algebras with respect to the family of cyclic subgroups \cite{LRRV2}. Using this result, forthcoming work of Moselle will calculate the TC of group algebras of elementary abelian $p$-groups.

In this paper, we complement these results about TC by proving an induction result for the K-theory of group algebras using recent advances in equivariant stable homotopy theory and calculations in modular representation theory. Our result by no means supplants TC as a calculational tool, but rather offers a way to extend results known by TC calculations. Using these, we enlarge the class of groups for which complete calculations of K-groups are available.

\begin{notation*}
In this work, $G$ is a finite group, $k$ is a field of prime characteristic $p$, and $S$ is a Sylow $p$-subgroup of $G$ with normalizer $N=N_S(G)$ and Weyl group $W=W_G(S)=N/S$. There are two hypotheses that appear frequently. We say 
\begin{itemize}
    \item $G$ is \emph{$p$-isolated} if $G$ contains no element of order $pq$ for $q$ a distinct prime from $p$, and
    \item $H\leq G$ is a \emph{trivial intersection} subgroup if $H\cap gHg\nv=e$ for all $g\notin N_G(H)$.
\end{itemize}
\end{notation*}

\begin{maintheorem} \label[maintheorem]{thm:main}
    \begin{maintheoremenum}
        \item  \label[maintheoremenumi]{thm:mainColim}
        If $k$ is finite and $G$ is $p$-isolated, then the $p$-adic K-groups of $kG$ satisfy
        \[
            \K_n(kG;\bbZ_p) \cong \colim_{G/H\in\cO_p(G)} \K_n(kH;\bbZ_p), \quad n>0,
        \]
        where $\cO_p(G)$ denotes the full subcategory of the orbit category spanned by the $p$-subgroups.
        \item \label[maintheoremenumi]{thm:mainOrb}
        If, in addition, any Sylow $p$-subgroup $S$ is trivial intersection, then the above colimit reduces to the coinvariants
        \[
            \K_n(kG;\bbZ_p) \cong \K_n(kS;\bbZ_p)/W, \quad n>0.
        \]
        with respect to the Weyl group $W$.
    \end{maintheoremenum}
\end{maintheorem}

This result concerns K-groups at the prime $p$, since $\K_*(kG)$ is equivalent to $\G_*(kG)$ after localization away from $p$ by the Dress--Kuku Theorem \cite{Dress--Kuku}. The G-theory of such group algebras is well-understood thanks to dévissage techniques in K-theory; see \cref{prop:Gtheory} for a description.

We highlight two extreme cases of \cref{thm:main}. First, if $p$ does not divide the order of $G$, then the only $p$-subgroup is trivial. The colimit decomposition reduces the $p$-adic K-groups of $kG$ in positive degrees to those of the ground field $k$, which are known to vanish by Quillen's computation. This is expected behavior in light of the equivalence $\K_*(kG)\xrightarrow{\sim} \G_*(kG)$ when $p$ does not divide the order of $G$, and the aforementioned G-theory decomposition.

Second, if $G$ is a $p$-group, then $S=G$, and our result provides no new information. Computing $\K_*(kG)$ when $G$ is a $p$-group is a difficult problem that requires different techniques, such as trace methods. Instead, we view \cref{thm:main} as a way to promote existing computations for $p$-groups to computations for groups containing these as Sylow subgroups. The state of knowledge for $p$-groups includes cyclic $p$-groups $\K_*(kC_{p^n})$ by work of Madsen \cite{Madsen:survey} and elementary abelian $p$-groups $\K_*(kC_p^n)$ by forthcoming work of Moselle.

\subsection{Overview of key ideas}

\subsubsection{The stable module category}

The central object is the \emph{stable module category} $\StMod(kG)$ of the group algebra $kG$. This is a stable $\infty$-category that has been extensively studied by representation theorists and homotopy theorists alike. The utility of the stable module category for K-theory is twofold: it admits multiplicative structure, and its K-theory captures the sought-after $p$-local summand of $\K_*(kG)$.

To introduce the stable module category, one starts from the inclusion of the category $\proj_G(k)$ of finite-dimensional $G$-representations over $k$ which are projective over $kG$ into the category $\rep_G(k)$ of all representations. The induced map on $\K_0$ recovers the well-studied \emph{Cartan map} from representation theory. One can equally well consider the induced map on K-theory spectra $\K(kG)\to \G(kG)$. Sarazola \cite{Maru:cotorsion} has shown that the cofiber of this map is the K-theory of a Waldhausen category which we call the \emph{stable module category}.

The stable module category can equally well be presented as a Verdier quotient of stable $\infty$-categories $\Perf(k)^{BG}/\Perf(kG)$ \cite{Rickard:stable-equivalence}. In this setup, the aforementioned cofiber sequence follows from the fact that K-theory is a localizing invariant \cite{BGT}. However, in this work, we employ the aforementioned explicit 1-categorical model. This has several advantages, such as arising from an economical model structure on the category of representations \cite{Hovey:modelcats}.

By replacing representations of $G$ over $k$ with coherent sheaves over a scheme $X$, one obtains an analogue to the stable module category  which goes by the name the \emph{singularity category} of $X$ \cite{Orlov:sg}. Briefly, this is because if $X$ is smooth, then all complexes of sheaves are perfect complexes, so this quotient detects the existence of singularities. Motivated by this, we term the K-theory of $\stmod(kG)$ the \emph{singular K-theory} $\Ksg(kG)$ of the group algebra $kG$.

\subsubsection{Mackey and Green functors}

Mackey functors arose in representation theory as an axiomatic framework for induction results \cite{Green:axiomatic,Dress:contributions}. Briefly, a \emph{Mackey functor} is a collection of abelian groups parameterized by subgroups of a finite group $G$ equipped with abstract induction and restriction operations mimicking those of representation rings. From the onset, multiplicative structure was recognized as indispensable to induction theory \cite{Dress:induction}. 
A \emph{Green functor} is a ring in Mackey functors, meaning in particular that each level is a commutative ring, and the induction and restriction operations respect this structure in a suitable sense. 

These notions have since been lifted to a homotopical context, and now form the underpinnings of equivariant stable homotopy theory. A foundational result, first shown by Guillou--May \cite{Guillou--May} and later revisited by Nardin \cite{Nardin:spectral-Mackey}, is that $G$-spectra are equivalent to \emph{spectral Mackey functors}, which are roughly Mackey functors valued in spectra. Just as K-theory spectra categorify Grothendieck groups, the equivariant K-theory $G$-spectra we study category these classical representation ring Mackey functors. 

A number of authors have considered constructions of K-theory $G$-spectra from categorical input. In their work on equivariant $A$-theory, Malkiewich--Merling \cite{Malkiewich--Merling1}, describe how a suitable collection of Waldhausen categories satisfying axioms analogous to a Mackey functor yield K-theory $G$-spectra. This was later revisited and simplified by Calle--Chan--Mejia \cite{Calle--Chan--Mejia}. Barwick has developed a different approach, which packages the data into a single bicartesian fibration \cite{Barwick:smfeakt}. This fibrational approach is the one we employ, as it was later extended by Barwick--Glasman--Shah to a machine which produces \emph{spectral Green functors} from suitable categorical input \cite{Barwick--Glasman--Shah}.

\subsubsection{The proof of \cref{thm:main}}

The central construction involves the three spectral Mackey functors which, informally, make the assignments
\begin{align*}
    \K_G(k) &: G/H \mapsto \K(kH), \\
    \K^G(k) &: G/H \mapsto \G(kH), \\
    \Ksg_G(k) &: G/H \mapsto \Ksg(kH).
\end{align*}
These fit into a cofiber sequence
\[
    \K_G(k) \to \K^G(k) \to \Ksg_G(k),
\]
which on $H$-fixed points recovers the aforementioned cofiber sequence of spectra
\[
    \K(kH) \to \G(kH) \to \K(\stmod(kH)).
\] 
The middle $G$-spectrum $\K^G(k)$ is the $G$-spectrum constructed by Merling \cite{Mona:Grings} associated to $k$ as a ring with trivial $G$-action. This is also called the K-theory of group actions by Barwick \cite{Barwick:smfeakt}. The first $G$-spectrum $\K_G(k)$ is the so-called ``coBorel version'' of equivariant algebraic K-theory which appears in Clausen--Mathew--Naumann--Noel \cite{CMNN2}. The last $G$-spectrum $\Ksg_G(k)$ has not appeared in the literature before, to our knowledge.

Our main technical innovation is showing that $\Ksg_G(k)$ is furthermore a spectral Green functor. We do so using the machinery developed by Barwick--Glasman--Shah, which they used to show that the K-theory of group actions $\K^G(k)$ is also a spectral Green functor \cite{Barwick--Glasman--Shah}. This is taken up in \cref{sec:Ksg}, where we lay the foundations for the the Cartan cofiber sequence of spectral Mackey functors. The main result is the following.

\begin{maintheorem} \label[maintheorem]{thm:Green}
    The spectral Mackey functor $\Ksg_G(k)$ admits the structure of a spectral Green functor, and the map $\K^G(k)\to \Ksg_G(k)$ is a map of spectral Green functors.
\end{maintheorem}

By the derived induction theory of Mathew--Naumann--Noel \cite{MNN2}, the spectral Green functor structure on $\Ksg_G(k)$ reduces the proof of \cref{thm:main} to proving an induction result for the ordinary Green functor $\ul S_k = \ul\pi_0 \Ksg_G(k)$. In \cref{sec:ind}, we review these concepts from induction theory, and prove the following.

\begin{maintheorem} \label[maintheorem]{thm:defbase}
    If $G$ is $p$-isolated, the Green functor $\ul S_k$ satisfies induction with respect to the family $\cF_p$ of $p$-subgroups, that is, the sum of induction maps 
    \[
        \bigoplus_{H\in\cF_p} \ul S_k(H) \xrightarrow{\bigoplus \Ind_H^G} \ul S_k(G).
    \]
    is surjective.
\end{maintheorem}

We assemble these ingredients into a proof of \cref{thm:main} in \cref{ssec:proof}. The key remaining piece is how singular K-theory of $kG$ captures the $p$-local summand of the ordinary K-groups of $kG$. Specifically, we show in \cref{cor:boundary} that for finite $k$, the boundary map associated to the Cartan cofiber sequence induces isomorphisms
\[
    \partial: \Ksg_{n+1}(kG) \xrightarrow{\sim} \K_n(kG;\bbZ_p), \quad n>0.
\]

\subsubsection{Applications} 

In \cref{sec:ex}, we give applications of \cref{thm:main}. We pay particular attention to the the situation where $G$ is $p$-isolated and the Sylow $p$-subgroup $S$ has order $p$. In this case, the trivial intersection condition is automatically satisfied, and the Weyl group action on K-groups  from \cref{thm:mainOrb} can be explicitly identified.

\begin{maintheorem} \label[maintheorem]{thm:Sylp}
    Let $G$ be a $p$-isolated finite group whose Sylow $p$-subgroups have order $p$. Then, for a finite field $k$,
    \[
        \K_n(kG;\bbZ_p) \cong \begin{cases}
            k^{\frac{p-1}{|W|}i} & n=2i-1>0, \\
            0 & \textrm{otherwise.}
        \end{cases}
    \]
\end{maintheorem}

Using \cref{thm:Sylp}, we calculate $p$-adic $K$-groups for group algebras of
\begin{itemize}
    \item the dihedral group $D_n$ for $n$ odd when $p=2$ (\cref{ex:dih2}),
    \item the dihedral group $D_p$ (\cref{ex:dihp}),
    \item the symmetric group $\Sigma_n$ when $n\leq p+1$ (\cref{ex:symp}),
    \item the alternating group $A_n$ when $n\leq p+2$ for odd $p$ (\cref{ex:altp}), 
    \item the affine general linear group $\AGL_1(p)\cong C_p\rtimes \Aut(C_p)$ (\cref{ex:AGL1p}), and
    \item the projective special linear group $\PSL_2(p)$ (\cref{ex:PSL2}).
\end{itemize}
As an illustration, we provide a complete integral calculation for of the K-groups of $\bbF_{p^r}\Sigma_p$ in \cref{cor:FpSp}. In \cref{ssec:bigS}, we state interesting further examples of $p$-isolated groups with larger Sylow $p$-subgroups, including
\begin{itemize}
    \item $A_4$ and $A_5$ at $p=2$ (\cref{ex:A45-2}),
    \item $A_6$ at $p=3$ (\cref{ex:A6-3}),
    \item $\Sigma_4$ at $p=2$ (\cref{ex:S4-2}), and
    \item $A_6$ and $\PSL_2(7)$ at $p=2$ (\cref{ex:A6-PSL27-2}).
\end{itemize}
For these, we cannot obtain explicit calculations of K-groups, but we can spell out how \cref{thm:main} reduces the calculations to $p$-subgroups.

\subsection*{Acknowledgements}

The author would like to thank the following people: Maxine Calle and David Chan for enlightening conversations about equivariant algebraic K-theory; Nir Gadish for suggesting the most economical form of the $p$-isolation condition; Peter Webb for pointing out the connection to the prime graph and Frobenius groups; Varinderjit Mann for help navigating relative categories; Kimball Strong for help navigating quasicategories; Cary Malkiewich for enlightening conversations about symmetric monoidal $\infty$-categories; and Isaac Moselle for sharing his calculations and feedback on an earlier draft. The author would particularly like to thank his advisor Inna Zakharevich for her guidance throughout the preparation of this work.

This work benefitted from a number of electronic tools. The online database \href{https://people.maths.bris.ac.uk/~matyd/GroupNames/}{GroupNames}, maintained by Tim Dokchitser, was convenient resource. The computer algebra system Magma \cite{magma} allowed for the efficient generation of examples.

During the preparation of this work, the author was partially supported by the NSF grant DMS-2052977 ``Trace Methods and Applications for Cut-and-Paste K-theory.''

\section{Singular K-theory as a spectral Green functor} \label{sec:Ksg}

The goal of this section is to prove \cref{thm:Green}, which states that the singular K-theory spectra $\Ksg(kH)$ for $H\leq G$ assemble into a spectral Green functor. We begin by briefly recalling this notion.

Denote by $\cA_G$ the \emph{effective Burnside $\infty$-category} of the finite group $G$ \cite[\S 3]{Barwick:smfeakt}. This is an $\infty$-categorical enhancement of the ordinary span category of finite $G$-sets. One can make sense of Mackey functors in a wide range of contexts by considering functors out of $\cA_G$. The most important example is the notion of a \emph{spectral Mackey functor}, a functor $\cA_G\to\Sp$. When doing K-theory, one constructs \emph{Mackey functors of Waldhausen $\infty$-categories} $\cA_G\to\Wald_\infty$, which yield spectral Mackey functors after postcomposition with the K-theory functor $\K:\Wald_\infty\to\Sp$.

Constructing functors $\cA_G\to\Wald_\infty$ is a delicate matter, as such functors encode an enormous amount of coherence data. Fortunately, work of Barwick provides a path forward. By the straightening/unstraightening correspondence, it suffices to construct a so-called Waldhausen cocartesian fibration over $\cA_G$. Barwick's unfurling construction \cite[\S 11]{Barwick:smfeakt} provides a means of constructing such a cocartesian fibration from that data of a so-called Waldhausen bicartesian fibration over $\Fin_G$. In \cref{ssec:Mackey}, we show that our Waldhausen categories fit into this framework.

Work of Barwick--Glasman--Shah constructs a symmetric monoidal structure on $\cA_G$, making is possible to speak of \emph{spectral Green functors} as monoidal functors $\cA_G^\otimes\to\Sp^\otimes$. As above, one constructs Green functors of Waldhausen $\infty$-categories $\cA_G^\otimes\to\Wald_\infty^\otimes$ then postcomposes with K-theory to obtain spectral Green functors. They show that to exhibit a Green functor of Waldhausen $\infty$-categories, it suffices to exhibit a so-called \emph{symmetric monoidal Waldhausen bicartesian fibration} \cite[\S 7]{Barwick--Glasman--Shah}. We do this in \cref{ssec:Green}, and thereby complete the proof of \cref{thm:Green}.

Throughout this section, the finite group $G$ and field $k$ are fixed. We highlight that the constructions in this section hold for arbitrary $k$, but unless the characteristic of $k$ divides $G$, the stable module category is the zero category.

\subsection{Waldhausen categories of representations} \label{ssec:Wald}

It will be convenient in the later sections to have a category of representations of any finite $G$-set, and not just merely subgroups of $H$.
This notion appears in work of Balmer \cite{Balmer:stacks}. In this section, we construct Waldhausen 1-category structures in the original sense of Waldhausen \cite{Waldhausen}.

\begin{definition} \label[definition]{def:Xrep}
    For a finite $G$-set $X$, the category $\Rep_X$ of \emph{$X$-representations} is the category of functors 
    \[
        X\sslash G\to\Vect(k),
    \]
    where $X\sslash G$ denotes the action groupoid of $X$, and $\Vect(k)$ denotes the category of $k$-vector spaces. An $X$-representation is
    \begin{itemize}
        \item \emph{free} if it is isomorphic to a sum of representable functors, and
        \item \emph{projective} if it is a summand of a free representation.
    \end{itemize} 
    We denote by $\Proj_X$ the full subcategory of those $X$-representations which are projective.
\end{definition}

\begin{notation}
    In general, we use lowercase to denote the the full subcategory spanned by finite-dimensional objects, e.g., $\rep_X = \vect(k)^{X\sslash G}$ is the category of $X$-representations which are pointwise finite-dimensional, and $\proj_X$ the subcategory thereof consisting of the finitely-generated projectives.
\end{notation}

The reader is assured that in the case $X=G/H$, these recover the notions of free and projective representations of the group $H$. Indeed, for any $x\in X$, there is an inclusion
\[
    \iota_x : BH \to X\sslash G
\]  
of the full subgroupoid spanned by $x$. When $X$ has a transitive $G$-action (i.e., $X$ is isomorphic to $G/H$), $\iota_x$ is an equivalence of groupoids, whence we obtain an equivalence of categories
\[
    \iota_x^* : \Rep_{G/H} \to \Vect(k)^{BH}=\Rep_H(k),
\]
where $\Rep_H(k)$ is the usual category of $H$-representations over $k$.

Even if $X$ is not transitive, if $\{x_i\}_{i\in I}$ is a collection of orbit representatives of finite $G$-set $X$, the inclusion of groupoids 
\[
    \coprod_i BG_{x_i} \xrightarrow{\coprod_i\iota_{x_i}} X\sslash G
\]
is similarly an equivalence, whence we obtain an equivalence 
\begin{equation} \label{eqn:basis}
    \Rep_X \xrightarrow{\prod_i \iota_{x_i}^*} \prod_i \Rep_{G_x}(k).
\end{equation}
In this way, we think of representations of finite $G$-sets $X$ as a ``basis-independent'' way of keeping track of collections of representations of subgroups of $G$.

\begin{proposition} \label[proposition]{prop:PRWald}
    For any finite $G$-set $X$, the category $\Proj_X$ and $\Rep_X$ each admit the structure of a Waldhausen category where 
    \begin{itemize}
        \item the cofibrations are the monomorphisms, and 
        \item the weak equivalences are the isomorphisms.
    \end{itemize}
\end{proposition}

\begin{proof}
    As an abelian category, $\Rep_X$ admits the structure of an exact category where any monomorphism is admissible. As the subcategory of projective objects in the abelian category $\Rep_X$, $\Proj_X$ admits an analogous exact structure. These two exact structures provide the stated Waldhausen structures by taking the weak equivalences to be the isomorphisms.
\end{proof}

We now describe the functoriality of the categories $\Rep_X$ and $\Proj_X$ in the $G$-set $X$. For a equivariant map $f:X\to Y$, there is a pullback functor $f^*:\Rep_Y\to\Rep_X$ given by precomposition with $f$. The functor $f^*$ admits a left adjoint we denote by $f_!:\Rep_X\to\Rep_Y$ \cite[Proposition 6.13]{Balmer:stacks}. For $V\in\Rep_X$, the $Y$-representation $f_!V$ is given by 
\[
    (f_!V)(y) = \bigoplus_{x\in f\nv(y)} V(x).
\]
In the case $f:G/H\to G/K$ is a map of transitive $G$-sets, $f^*$ and $f_!$ recover the familiar restriction and induction functors $\Res_H^K$ and $\Ind_H^K$, respectively. The functors $f^*$ and $f_!$ each preserve projectives, and thus the adjunction $f_!\dashv f^*$ restricts to the subcategories of projectives. 

\begin{proposition} \label[proposition]{prop:PRexact}
    For any map of finite $G$-sets $f:X\to Y$, the functors $f^*$ and $f_!$ are exact functors of Waldhausen categories with respect to the Waldhausen category structures of \cref{prop:PRWald}
\end{proposition}

\begin{proof}
    The pullback functor $f^*$ is exact as a functor of abelian categories, and thus preserves monomorphisms and pushouts. Given the above description as a sum over preimages, the pushforward $f_!$ preserves monomorphisms, and additionally preserves pushouts as a left adjoint.
\end{proof}

We now turn to constructing a Waldhausen category presenting the stable module category. Say that a map $\varphi:V\to W$ in $\Rep_X$ is a \emph{stable equivalence} if there exists a \emph{stable inverse} $\tilde\varphi:W\to V$ such that $\varphi\tilde\varphi-\id_W$ and $\tilde\varphi\varphi-\id_V$ each factor through projectives. We denote by $\cW_X$ the collection of stable equivalences in $\Rep_X$. 

We briefly digress to show that even more structure holds than Waldhausen category structure. The category $\Rep_X$ is a \emph{Frobenius} abelian category for any $X$, meaning that the classes of injective and projective objects coincide \cite[Example 6.5]{Balmer:stacks}. In his work on cotorsion pairs, Hovey \cite{Hovey:cotorsion} showed that such a category has a model structure in the sense of Quillen where 
\begin{itemize}
    \item the cofibrations are the monomorphisms,
    \item the fibrations are the epimorphisms, and
    \item the weak equivalences are the stable equivalences.
\end{itemize} 
We obtain the following by forgetting structure (cf. \cite[Lemma 3.1]{Dugger--Shipley:derived-equivalence}).

\begin{proposition} \label[proposition]{prop:SWald}
    For any finite $G$-set $X$, the category $\Rep_X$ admits the structure of a Waldhausen category where 
    \begin{itemize}
        \item the cofibrations are the monomorphisms, and 
        \item the weak equivalences are the stable equivalences $\cW_X$.
    \end{itemize}
\end{proposition}

\begin{proposition} \label[proposition]{prop:Sexact}
    For any map of finite $G$-sets $f:X\to Y$, the functors $f^*$ and $f_!$ are exact functors with respect to the Waldhausen category structures of \cref{prop:SWald}
\end{proposition}

\begin{proof}
    In light of the proof of \cref{prop:PRexact}, what is left to show is that these functors preserve stable equivalences. This follows from the fact that both $f^*$ and $f_!$ preserve projectives, and thus preserve the property of factoring through a projective.
\end{proof}

We call the Waldhausen category resulting from \cref{prop:SWald} the \emph{stable module category} $\StMod_X$. In summary, there is a sequence of Waldhausen categories and exact functors
\[
    \Proj_X \to \Rep_X \to \StMod_X
\]
associated to any finite $G$-set $X$. When taking K-theory, we pass to the full subcategories finite-dimensional objects. Sarazola has shown that this yields a cofiber sequence.

\begin{proposition}[{\cite[Example 8.6]{Maru:cotorsion}}] \label[proposition]{Maru}
    For any finite $G$-set $X$, there is a cofiber sequence of K-theory spectra 
    \[
        \K(\proj_X) \to \K(\rep_X) \to \K(\stmod_X).
    \]
\end{proposition}

In the case where $X$ is the transitive orbit $G/H$, this recovers the Cartan cofiber sequence of spectra $\K(kH)\to \G(kH)\to \K(\stmod(kH))$ for the group $H$.

\subsection{Mackey functoriality} \label{ssec:Mackey}

In this section, we stitch the categories of the previous section together into the data of Waldhausen bicartesian fibrations, in the sense of Barwick \cite[Definition 10.3]{Barwick:smfeakt}. As a first step, we translate from the Waldhausen 1-categories into quasicategories via relative nerves. 

The particular relative nerve we employ is based on marked simplicial sets. A \emph{marked simplicial set} is a simplicial set equipped with a collection of marked edges containing all degenerate edges. These form a category $\sSet^+$ where the morphisms are those simplicial maps which preserve the markings. A special case of Lurie's cartesian model structure endows $\sSet^+$ with a model structure in which the functor forgetting markings $\sSet^+\to\sSet$ is a right Quillen equivalence to the Joyal model structure \cite[Proposition 3.1.5.3]{HTT}. We use that for any 1-category $\cC$ equipped with a collection of weak equivalences $\cW$, the fibrant replacement of $(\N\cC,\cW)$ in the model structure on $\sSet^+$ yields a quasicategory $\N\cC[\cW\nv]$, the localization of $\cC$ at the collection $\cW$. 

Barwick has introduced the notion of a \emph{Waldhausen $\infty$-category}. This is a \emph{pair} consisting of an $\infty$-category and a subcategory specifying the cofibrations (or in his terminology, ingressive morphisms), satisfying axioms analogous to Waldhausen's definition \cite[Definition 2.7]{Barwick:Waldhausen}. Associated to any Waldhausen 1-category $\cC$ with cofibrations $c\cC$ and weak equivalences $w\cC$, there is a Waldhausen $\infty$-category structure on the relative nerve $\N\cC[w\cC\nv]$, where the subcategory of cofibrations is the smallest subcategory containing $c\cC$ \cite[Example 2.12]{Barwick:Waldhausen}. Barwick shows that his K-theory construction for this induced Waldhausen $\infty$-cateogry agrees with Waldhausen's definition \cite[Corollary 10.16]{Barwick:Waldhausen}.

We define three functors $P,R,S:\Fin_G^\op\to\sSet^+$ by assigning 
\begin{align*}
    P(X) &= (\N\proj_X,\textrm{iso.}), \\
    R(X) &= (\N\rep_X,\textrm{iso.}), \\
    S(X) &= (\N\rep_X,\cW_X),
\end{align*}
to any finite $G$-set, and assigning the pullback functors $f^*$ to any morphism $f$ of finite $G$-sets. As objects in $(\sSet^+)^{\Fin_G^\op}$, $P$ and $R$ are fibrant in the projective model structure, as they are levelwise fibrant. Indeed, the marked edges are already equivalences. In the case of $S$, fibrant replacement yields a functor $\wt S$ with
\[
    \wt S(X) = (\N\Rep_X[\cW_X\nv],\cW_X).
\]

A special case of the straightening/unstraightening correspondence is given by Lurie's \emph{marked relative nerve} functor \cite[Definition 3.2.5.12]{HTT} A contravariant version of the marked relative nerve gives a right Quillen equivalence 
\[
    \N^+_{(-)}\Fin_G :(\sSet^+)^{\Fin_G^\op} \to \sSet^+_{/\N\Fin_G}
\]
from the projective model structure on $(\sSet^+)^{\Fin_G^\op}$ to the cartesian model structure on the overcategory $\sSet^+_{/\N\Fin_G}$ \cite[Proposition 3.2.5.18(2)]{HTT}. By applying the marked relative nerve to the functors $P$, $R$, and $\wt S$, we obtain quasicategories over $\N\Fin_G$ we denote by
\begin{align*}
    \cP &= \N^+_P\Fin_G \xrightarrow{u_\cP} \N\Fin_G, \\
    \cR &= \N^+_R\Fin_G \xrightarrow{u_\cR} \N\Fin_G, \\
    \cS &= \N^+_{\wt S}\Fin_G \xrightarrow{u_\cS} \N\Fin_G.
\end{align*}
Since the marked relative nerve is right Quillen, these are fibrant in the cartesian model structure on marked simplicial sets over $\N\Fin_G$. As such, they are each cartesian fibrations over $\N\Fin_G$ \cite[Proposition 3.1.4.1]{HTT}. 

We reintroduce the Waldhausen category structure by equipping $\cP$, $\cR$, and $\cS$ with the pair structure in which the ingressive morphisms are morphisms which are cofibrations in the Waldhausen category structure on each fibrer. We turn to showing this data suffices to present Mackey functors of Waldhausen categories. This entails showing that our cartesian fibrations satisfy Barwick's definition, which we now present. 

While Barwick defines Mackey functors with respect to any \emph{disjunctive triple}, we are only interested in the case of $\Fin_G$ here, which simplifies some of the conditions. We follow the numbering of the original.

\begin{definition}[{\cite[Definition 10.3]{Barwick:smfeakt}}] 
    \label[definition]{def:Wbicart}
    A functor of pairs $u:\cX\to\N\Fin_G^\flat$ is a \emph{Waldhausen bicartesian fibration} provided it satisfies the following properties.
    \begin{definitionenum}
        \item \label[definitionenumi]{def:WB1}
        The functor $u$ is a cocartesian fibration. 
        \item \label[definitionenumi]{def:WB2}
        The functor $u$ is a cartesian fibration.
        \item \label[definitionenumi]{def:WB3}
        The \emph{Beck--Chevalley condition} holds: for any pullback square 
        \[\begin{tikzcd}
            X \arrow[r,"i"] \arrow[d,"q"'] & X' \arrow[d,"r"] \\
            Y \arrow[r,"j"'] & Y'
        \end{tikzcd}\]   
        of finite $G$-sets, the natural map $i_!q^* \to r^*j_!$ is an equivalence.
        \setcounter{definitionenumi}{4}
        \item \label[definitionenumi]{def:WB5}
        For any map $f:X\to Y$ of $G$-sets, the pushforward functor $f_!:\cX_X\to\cX_Y$ is an exact functor of Waldhausen $\infty$-categories.
        \item \label[definitionenumi]{def:WB6}
        For any map $f:X\to Y$ of $G$-sets, the pullback functor $f^*:\cX_Y\to\cX_X$ is an exact functor of Waldhausen $\infty$-categories.
        \item \label[definitionenumi]{def:WB7}
        For any finite set $I$ and $I$-indexed collection $X_i$ of finite $G$-sets with coproduct $X=\coprod_I X_i$, the pullback functors
        \[
            j_i^*: \cX_X \to \cX_{X_i}
        \]
        along the inclusions $j_i:X_i\to X$ together exhibit $\cX_X$ as the direct sum $\bigoplus_I \cX_{X_i}$.
    \end{definitionenum}
\end{definition}

\begin{proposition} \label[proposition]{prop:Wbicart}
    Each of the cartesian fibrations $\cP$, $\cR$, and $\cS$ over $\N\Fin_G$ is a Waldhausen bicartesian fibration.
\end{proposition}

\begin{proof}
    We address each of the conditions in \cref{def:Wbicart}. 

    By construction, all three are cartesian fibrations, which gives \cref{def:WB2}. Additionally, all three are cocartesian fibrations, because each of the pullback maps $f^*$ admit left adjoints (cf. \cite[Proposition 4.1.4.17(1)]{HA}), which gives condition \cref{def:WB1}. Balmer has shown the Beck--Chevalley condition (\cref{def:WB3}) holds for the categories $\rep_X$ and $\stmod_X$ \cite[Lemma 7.10]{Balmer:stacks}, and it additionally holds for the categories $\proj_X$ since the (co)units of the adjunctions restrict to the category of projectives. The conditions \cref{def:WB5,def:WB6} follow from \cref{prop:PRexact} for $\cP$ and $\cR$, and from \cref{prop:Sexact} for $\cS$. Lastly, we address \cref{def:WB7}. The formation of action groupoids from finite $G$-sets preserves coproducts, so we have
    \begin{align*}
        \rep_X = \Fun(X\sslash G, \vect(k)) 
        \simeq \Fun\Big( \coprod_{i\in I} X_i \sslash G , \vect(k) \Big)
        \simeq \prod_{i\in I} \Fun(X_i \sslash G, \vect(k))
        = \prod_{i\in I} \rep_{X_i}.
    \end{align*} 
    The cases of $\cP$ and $\cS$ are analogous.
\end{proof}

Barwick shows that Waldhausen bicartesian fibrations yield spectral Mackey functors after applying his unfurling construction \cite[\S 11.8]{Barwick:smfeakt}. This associates to a Waldhausen bicartesian fibration $u:\cX\to\N\Fin_G$ its \emph{unfurling}
\[
    \Upsilon(u): \Upsilon(\cX/\N\Fin_G) \to \cA_G
\]
whose straightening is the Mackey functor 
\[
    \mathbf{M}(u): \cA_G \to \Wald_\infty.
\]
We denote the spectral Mackey functors resulting from our above fibrations by 
\begin{align*}
    \K_G(k) &= \K\circ \mathbf{M}(u_\cP) \\
    \K^G(k) &= \K\circ \mathbf{M}(u_\cR) \\
    \Ksg_G(k) &= \K\circ \mathbf{M}(u_\cS)
\end{align*}
The resulting sequence of spectral Mackey functors
\begin{equation} \label{eqn:cofiber}
    \K_G(k) \to \K^G(k) \to \Ksg_G(k),
\end{equation}
is a cofiber sequence, since it is a cofiber sequence pointwise by \cref{Maru}. This is a lift of the Cartan cofiber sequence to a cofiber sequence of spectral Mackey functors.

\subsection{Green functoriality} \label{ssec:Green}

In this section, we promote $u_\cR$ and $u_\cS$ to have the structure of symmetric monoidal Waldhausen bicartesian fibrations. In the case of $u_\cR$ this structure is not new; it is a special case of the ``K-theory of group actions'' which Barwick--Glasman--Shah prove Green functor structure for in their work \cite[Proposition 8.2]{Barwick--Glasman--Shah}. 

We will need an alternate indentifaction of the marked relative nerve $\cR=\N^+_R\Fin_G$. By applying the 1-categorical Grothendieck construction to the functor $R:\Fin_G^\op\to\Cat$, we obtain a cartesian fibration $\rep\to\Fin_G$ whose fiber over $X$ is exactly the category $\rep_X$ described above. Explicitly, the category $\rep$ has 
\begin{itemize}
    \item objects: pairs $(X,V)$ consisting of a finite $G$-set $X$ with a finite-dimensional $X$-representation $V:X\sslash G\to \vect(k)$, and
    \item morphisms $(X,V)\to(Y,W)$: pairs $(f,\varphi)$ consisting of an equivariant map $f:X\to Y$ and a map of $X$-representations $\varphi : V \to f^*W$.
\end{itemize}
In the construction of $\cR$, we only marked the isomorphisms of the 1-categories $\rep_X$, so there is an isomorphism $\N\rep\cong \cR=\N^+_R\Fin_G$.

We use this identification to import monoidal structure. Indeed, each of the fibers $\rep_X$ admits a symmetric monoidal structure via the pointwise product given by
\[
    (V\otimes W)(x) = V(x)\otimes W(x), \quad x\in X.
\]
With a choice of orbit representatives $\{x_i\}_{i\in I}$ of $X$, the equivalence 
\[
    \rep_X \to \prod_{i\in I} \rep_{G_{x_i}}(k)
\]
of \cref{eqn:basis} is monoidal, where each category $\rep_H(k)$ has the monoidal structure coming from $\otimes_k$. The total category $\rep$ admits a symmetric monoidal structure $\boxtimes$ given by
\[
    (X,V) \boxtimes (Y,W) = (X\times Y, \pi_1^*V \otimes \pi_2^*W),
\]
where $\pi_1$ and $\pi_2$ are the projections of $X\times Y$ onto $X$ and $Y$, respectively. We refer to $\boxtimes$ as the \emph{external product}. The aforementioned \emph{internal product} on each of the fibers $\rep_X$ is recovered via the diagonal map as the composite
\[
    \rep_X \times \rep_X \xrightarrow{\boxtimes} \rep_{X\times X} \xrightarrow{\Delta^*} \rep_X.
\]
With the external product, the functor $\rep\to\Fin_G$ is symmetric monoidal. This makes $\rep$ a symmetric monoidal bifibration in the sense of Shulman \cite{Shulman:monfib}.

By the construction of $\cS$ as a fiberwise localization of $\cR$, it follows (cf. the opposite of \cite[\href{https://kerodon.net/tag/02LW}{Tag 02LW}]{kerodon}) that 
\[
    \cS\simeq\cR[\cW\nv]=\N\rep[\cW\nv]
\]
where $\cW=\coprod_X \cW_X$ is the collection of fiberwise stable equivalences in $\rep$. Our next task to show that this localization is in fact a monoidal localization. We prove some lemmas towards this end.

\begin{lemma} \label[lemma]{lem:projbox}
    For any $X\in\Fin_G$, if $U$ is a projective $X$-representation and $V$ is any $X$-representation, then $U\otimes V$ is projective.
\end{lemma}

\begin{proof}
    Since tensor products preserve summands, it suffices to show that if $U$ is free, then $U\otimes V$ is free. By choosing orbit representatives as in \cref{eqn:basis}, it suffices to prove the corresponding statement in $\Rep_H(k)$ for any $H\leq G$. There, the free representations are the sums of copies of the regular representation $kH\cong\Ind_e^H(k)$. The result follows then from Frobenius reciprocity:
    \[
        V\otimes\Ind_e^H(k) \cong \Ind_e^H(\Res_e^H(V)\otimes k)
    \]
    is free, as a representation induced from the trivial group.
\end{proof}

\begin{lemma} \label[lemma]{lem:boxtimes}
    The external product $\boxtimes$ on $\rep$ preserves the class of stable equivalences separately in each variable.
\end{lemma}

\begin{proof}
    To show that $\boxtimes$ preserves stable equivalences, it suffices to show that $\boxtimes$ preserves those maps which factor through projectives. This follows from \cref{lem:projbox}.
\end{proof}

By work of Hinich \cite[Proposition 3.2.2]{Hinich:localization}, \cref{lem:boxtimes} is sufficient to conclude that the localization $\cS\simeq \cR[\cW\nv]$ is a monoidal localization. We summarize the resulting structure in the following proposition.

\begin{proposition} \label[proposition]{prop:boxtimes}
    The $\infty$-category $\cS$ admits a symmetric monoidal structure $\cS^\boxtimes$, such that the functor 
    \[
    \cR^\boxtimes \to \cS^\boxtimes \simeq \cR[\cW\nv]^\boxtimes
    \] 
    is symmetric monoidal.
\end{proposition}

We turn to showing this monoidal structure is sufficient to exhibit the Green functoriality on the Mackey functors associated to $\cR$ and $\cS$ from the previous section. This entails showing they satisfy the following definitoin of Barwick--Glasman--Shah.

\begin{definition}[{\cite[Definition 7.3]{Barwick--Glasman--Shah}}]
    \label[definition]{def:SMWbicart}
    A \emph{symmetric monoidal Waldhausen bicartesian fibration} is a functor $u_\boxtimes:\cX_\boxtimes\to {\N\Fin_G}_\times$ satisfying the following properties.
    \begin{definitionenum}
        \setcounter{definitionenumi}{1}
        \item \label[definitionenumi]{def:SM2}
        The functor $u_\boxtimes$ is a cartesian fibration. 
        \item \label[definitionenumi]{def:SM3}
        The composite 
        \[
            \cX_\boxtimes \to {\N\Fin_G}_\times \to \N\Fin_*^\op
        \]
        exhibits $\cX_\boxtimes$ an an $\infty$-anti-operad.
        \item \label[definitionenumi]{def:SM4}
        The fiber $u:\cX\to\N\Fin_G$ over $\langle 1 \rangle \in \N\Fin_*^\op$ is a Waldhausen bicartesian fibration in the sense of \cref{def:Wbicart}.
    \end{definitionenum}
\end{definition}

To obtain this from our existing symmetric monoidal structures, we make use of the \emph{dual} of a cocartesian fibration, a construction introduced by Barwick--Glasman--Nardin \cite{Barwick--Glasman--Nardin:dual}. This associates to a cocartesian fibration of $\infty$-categories $u:\cX\to\cB$ a cartesian fibration $u^\wed:\cX^\wed\to\cB^\op$ which straightens to the same functor $\cB\to\Cat_\infty$ as $u$ does. If an $\infty$-category $\cC$ has a symmetric monoidal structure given by the $\infty$-operad $\cC^\otimes\to\N\Fin_*$, the dual $\cC^{\otimes,\wed}\to\N\Fin_*^\op$ is an $\infty$-anti-operad, as the opposite of the $\infty$-operad presenting the symmetric monoidal structure on $\cC^\op$.

\begin{proof}[Proof of \cref{thm:Green}]
    It suffices to show that $u_\cR^\wed$ and $u_\cS^\wed$ satisfy the conditions of \cref{def:SMWbicart}. Indeed, Barwick--Glasman--Shah show that the additional data of a symmetric monoidal Waldhausen bicartesian fibration equips the unfurling with the structure of a Green functor \cite[Corollary 7.8.1]{Barwick--Glasman--Shah}, and the underlying Waldhausen bicartesian fibrations, the fibers over $\langle 1\rangle\in\N\Fin_G$, are unchanged by the dual.

    These fibers were shown to be Waldhausen bicartesian fibrations in \cref{prop:Wbicart}, which gives condition \cref{def:SM4}. For condition \cref{def:SM3}, $\cR^\boxtimes\to\N\Fin_G^\times$ and $\cS^\boxtimes\to\N\Fin_G^\times$ are symmetric monoidal, so their duals
    \begin{align*}
        \cR_\boxtimes = \cR^{\boxtimes,\wed} &\xrightarrow{u_\cR^\wed} \N\Fin_G^{\times,\wed} = {\N\Fin_G}_\times \\
        \cS_\boxtimes = \cS^{\boxtimes,\wed} &\xrightarrow{u_\cS^\wed} \N\Fin_G^{\times,\wed} = {\N\Fin_G}_\times
    \end{align*}
    yield the desired maps of $\infty$-anti-operads. Lastly, to obtain the condition \cref{def:SM2}, it suffices to check that cartesian arrows are preserved by $\boxtimes$ (cf. \cite[Lemma 2.10]{Ramzi:monoidal}), which follows from \cref{lem:boxtimes}.
\end{proof}

\section{Induction theorems} \label{sec:ind}

Classically, the \emph{Cartan homomorphism} is the map 
\[
    \K_0(kG) = P_k(G) \to R_k(G) = \G_0(kG)
\]
from the Grothendieck group of projective $kG$-modules to the representation ring of $G$ over $k$. With the projective indecomposables and irreducibles as bases for $P_k(G)$ and $R_k(G)$, respectively, the resulting \emph{Cartan matrix} encodes the irreduible representations appearing as the composition factors of each projective indecomposable. 

From the previous section, we have a cofiber sequence of spectral Mackey functors \cref{eqn:cofiber}
\[
    \K_G(k) \to \K^G(k) \to \Ksg_G(k),
\]
where $\K_G(k)\to \K^G(k)$ is a lift of the Cartan homomorphism to spectral Mackey functors. We term this cofiber sequence the \emph{Cartan cofiber sequence}. If we define Mackey functors 
\begin{align*}
    \ul P_k &= \ul\pi_0 \K_G(k), \\
    \ul R_k &= \ul\pi_0 \K^G(k), \\
    \ul S_k &= \ul\pi_0 \Ksg_G(k),
\end{align*}
then the Cartan cofiber sequence \cref{eqn:cofiber} yields the short exact sequence of Mackey functors
\[
    0\to \ul P_k \to \ul R_k \to \ul S_k\to 0.
\]
Indeed, the Cartan homomorphism is known to be injective \cite[\S 16.1, Corollary 1 to Theorem 35]{Serre:lrfg}. As a result, for each $H\leq G$, the abelian group $\ul S_k(H)$ is the cokernel of the Cartan matrix for $kH$.

By \cref{thm:Green}, $\ul R_k \to \ul S_k$ is a map of Green functors, exhibiting $\ul S_k$ as the quotient of the Green functor $\ul R_k$ by the Green ideal $\ul P_k$. In this section, we investigate induction theorems for both the Green functor $\ul S_k$ and the spectral Green functor $\Ksg_G(k)$ using the derived induction theory framework developed by Mathew--Naumann--Noel \cite{MNN2}. In \cref{ssec:defbase}, we prove \cref{thm:defbase}, an induction result for the Green functor $\ul S_k$, and in \cref{ssec:proof}, we use this to prove \cref{thm:main}.

\subsection{Induction theorems on \texorpdfstring{$\K_0$}{K0}} \label{ssec:defbase}

We introduce some terminology. A \emph{family} of subgroups of $G$ is a collection of subgroups closed under subconjugacy. A Green functor $\ul R$ \emph{satisfies induction with respect to $\cF$} if the sum of induction maps
\[
    \Ind^G_\cF = \bigoplus_{H\in\cF}\Ind_H^G: \bigoplus_{H\in\cF} \ul R(H) \to \ul R(G)
\]
is surjective. The \emph{defect base} of $\ul R$ is the minimal family which $\ul R$ satisfies induction with respect to. This minimal family is known to be unique \cite[Lemma 3.11]{Green:axiomatic}.

In this language, the Brauer induction theorem states that for $k$ a sufficently large field, the defect base of $\ul R_k$ is the family $\cE$ of Brauer elementary subgroups \cite[\S 17.2, Corollary to Theorem 39]{Serre:lrfg}. A subgroup $H$ is \emph{Brauer elementary} if it is of the form $C\times Q$ where $Q$ is an $q$-group for some prime $q$ and $C$ is a cyclic $q'$-group. A field $k$ is \emph{sufficiently large} in the sense of Serre \cite{Serre:lrfg} if it contains all $m$th roots of unity, where $m$ is the exponent of $G$.

In this section, we prove \cref{thm:defbase}, which states that if $G$ is $p$-isolated, the defect base of the Green functor $\ul S_k$ is the family of $p$-subgroups. We first prove this for $k$ sufficiently large before extending to general fields $k$.

\begin{lemma} \label[lemma]{lem:elem}
    Suppose $k$ is a sufficiently large field of characteristic $p$. For any finite group $G$, $\ul S_k$ satisfies induction with respect to the family $\cE$ of Brauer elementary subgroups.
\end{lemma}

\begin{proof}
    Consider the map of short exact sequences
    \[ \begin{tikzcd}
        \ul P_k(G) \arrow[r, hook] & 
        \ul R_k(G) \arrow[r, two heads] & 
        \ul S_k(G) \\
        \displaystyle\bigoplus_{H\in\cE} \ul P_k(H) \arrow[r, hook] \arrow[u, two heads, "\Ind_\cE^G"] & 
        \displaystyle\bigoplus_{H\in\cE} \ul R_k(H) \arrow[r, two heads] \arrow[u, two heads, "\Ind_\cE^G"] &
        \displaystyle\bigoplus_{H\in\cE} \ul S_k(H) \arrow[u, "\Ind_\cE^G"]
    \end{tikzcd} \]
    induced by induction from the family $\cE$. The first two vertical arrows are surjections by the Brauer induction theorem. It follows that the last vertical arrow is also a surjection.
\end{proof}

We now show that the Green functor $\ul S_k$ satisfies a stronger induction result. Recall that the levels of the Green functor $\ul S_k$ are the cokernels of Cartan matrices. We make use of the following standard facts about Cartan matrices:
\begin{itemize}
    \item If $G$ is a $p'$-group, then the Cartan matrix of $kG$ is an identity matrix \cite[\S 15.5]{Serre:lrfg}.
    \item If $G$ is a $p$-group, then the Cartan matrix of $kG$ is the $1\times 1$ matrix $C=[|G|]$ \cite[\S 15.6]{Serre:lrfg}.
\end{itemize}

\begin{proposition} \label[proposition]{prop:deflarge}
    Suppose $k$ is a sufficiently large field of characteristic $p$ and $G$ is a $p$-isolated group. Then, $\ul S_k$ satisfies induction with respect to the family of $p$-subgroups.
\end{proposition}

\begin{proof}
    Because $G$ is $p$-isolated, Brauer elementary subgroups of $G$ are either $p$-groups or have order coprime to $p$. The Green functor $\ul S_k$ vanishes on subgroups coprime to $p$, so in fact, $\ul S_k$ satisfies induction with respect to the family of $p$-subgroups by \cref{lem:elem}.
\end{proof}

We extend this result to an arbitrary field $k$ of characteristic $p$, which yields \cref{thm:defbase}. 

\begin{proof}[Proof of \cref{thm:defbase}]
    Embed $k$ into a sufficiently large field $k'$ and consider the diagram
    \[ \begin{tikzcd}
        \ul S_k(G) \arrow[r, "\iota_S"] &
        \ul S_{k'}(G) \\
        \displaystyle\bigoplus_{H\in \cF_p} \ul S_k(H) \arrow[r, "\cong"] \arrow[u, "\Ind_{\cF_p}^G"] &
        \displaystyle\bigoplus_{H\in \cF_p} \ul S_{k'}(H) \arrow[u,two heads, "\Ind_{\cF_p}^G"]
    \end{tikzcd} \]
    where $\cF_p$ is the family of $p$-subgroups. The right vertical arrow is surjective by \cref{prop:deflarge}. The result follows if we show $\iota:\ul S_k(G)\to \ul S_{k'}(G)$ is injective. To this end, consider the diagram induced by extension of scalars
    \[ \begin{tikzcd}
        \ul P_k(G) \arrow[r, hook] \arrow[d, hook, "\iota_P"] &
        \ul R_k(G) \arrow[r, two heads] \arrow[d, hook, "\iota_R"] & 
        \ul S_k(G) \arrow[d, "\iota_S"] \\
        \ul P_{k'}(G) \arrow[r, hook] &
        \ul R_{k'}(G) \arrow[r, two heads] & 
        \ul S_{k'}(G).
    \end{tikzcd} \]
    The first two arrows are known to be split injective \cite[\S 14.6]{Serre:lrfg}. To show injectivity of the last arrow $\iota_S$, the snake lemma gives an exact sequence containing
    \[
        0 \to \ker(\iota_S) \to \coker(\iota_P).
    \]
    Since $\coker(\iota_P)$ is a free abelian group and $\ker(\iota_S)$ is a subgroup of the finite group $\ul S_k(G)$, it must be that $\ker(\iota_S)=0$.
\end{proof}

\subsection{Consequences for higher K-theory} \label{ssec:proof}

In this section, we assemble the preceeding results into a proof of \cref{thm:main}. 

We begin by clarifying how singular K-theory captures the K-theory of $kG$ at the prime $p$. After $p$-completion, the Cartan cofiber sequence \cref{eqn:cofiber} takes the form 
\[
    \K_G(k;\bbZ_p) \to \K^G(k;\bbZ_p) \to \Ksg_G(k;\bbZ_p).
\]
The middle term $K^G(k;\bbZ_p)$ has as its $H$-fixed points the $p$-complete $G$-theory spectrum 
\[
    \K^G(k;\bbZ_p)^H \simeq \G(kG;\bbZ_p).
\]
We will make use of the following decomposition of the G-theory of group algebras.

\begin{proposition} \label[proposition]{prop:Gtheory}
    We have 
    \[
        \G(kG) \simeq \bigoplus_{V} \K(\End(V))
    \]
    where the sum runs over the irreducible representations $V$ of $G$ over $k$.
\end{proposition}

\begin{proof}
    This is a standard Devissage argument \cite[Application V.4.3]{kbook}.
\end{proof}

If $k$ is finite, then the group algebra $kG$ of a finite group $G$ is a finite ring. Kuku has shown that the higher K-groups of finite rings are finite abelian groups \cite[Proposition IV.1.16]{kbook}. It follows that the $p$-adic K-groups are the $p$-local components of these finite abelian groups.

\begin{corollary} \label[corollary]{cor:boundary}
    For finite $k$, the boundary map associated to the Cartan cofiber sequence \cref{eqn:cofiber} induces a an isomorphism of Mackey functors
    \[
        \ul\pi_{n+1} \Ksg_G(k) \xrightarrow{\sim} \ul\pi_{n} \K_G(k;\bbZ_p), \qquad n>0.
    \]
\end{corollary}

\begin{proof}
    If $k$ is finite, then the endomorphism rings appearing in the decomposition of \cref{prop:Gtheory} are finite-dimensional division algebras over $k$. A finite-dimensional division algebra must be a finite fields by Wedderburn's little theorem. It follows that $\ul\pi_i K^G(k;\bbZ_p)=0$ for $i>0$ by Quillen's calculation of the K-theory of finite fields \cite{Quillen:finite-field}. From the long exact sequence associated to the $p$-complete Cartan cofiber sequence, we have 
    \[
        \ul\pi_n \Ksg_G(k;\bbZ_p) \xrightarrow{\sim} \ul\pi_{n-1}\K_G(k;\bbZ_p).
    \]
    By the Dress--Kuku theorem \cite{Dress--Kuku}, the Cartan map induces an isomorphism  
    \[
        \ul\pi_* \K_G(k)[1/p] \xrightarrow{\sim} \ul\pi_*\K^G(k)[1/p].
    \]
    Thus, the long exact sequence associated to the Cartan cofiber sequence after localization away from $p$ yields $\ul\pi_*\Ksg_G(k)[1/p]\simeq 0$. We conclude $\ul\pi_{n+1} \Ksg_G(k;\bbZ_p) \cong \ul\pi_{n+1} \Ksg_G(k)$, which completes the proof.
\end{proof}

With this result in hand, we can now deduce the main induction theorem for K-theory from the induction theorem on singular K-theory implied by \cref{thm:defbase}. First, we recall the \emph{orbit category} $\cO(G)$ is the full subcategory of $\Fin_G$ spanned by the transitive $G$-sets. This is equivalent to the category whose
\begin{itemize}
    \item objects are subgroups $H$ of $G$, and 
    \item set of morphisms $H\to K$ is the set $K\backslash N_G(H,K)$, where $N_G(H,K)$ is the \emph{transporter}
    \[
        N_G(H,K) = \{ g\in G: gHg\nv\leq K \}.
    \]
\end{itemize}
Indeed, any $g\in N_G(H,K)$ determines an equivariant map of coset spaces $f:G/H\to G/K$ by $f(aH)=agK$, and all such maps arise this way. Furthermore, if $g'=gy$ for some $y\in K$, then $g$ and $g'$ evidently induce the same map $G/H\to G/K$. 

Given a family $\cF$, we write $\cO_\cF(G)$ for the full subcategory spanned by those transitive sets with isotropy in $\cF$. In the case where $\cF_p$ is the family of $p$-subgorups, we write $\cO_p(G)=\cO_{\cF_p}(G)$. 

\begin{proposition} \label[proposition]{prop:splvl}
    If $G$ is $p$-isolated and $k$ is finite,
    \[
        \Ksg(kG) \simeq \hocolim_{G/H\in\cO_p(G)} \Ksg(kH).
    \]
\end{proposition}

\begin{proof}
    By \cref{thm:Green}, $\Ksg_G(k)$ is a $G$-ring spectrum, and by \cref{thm:defbase}, the Green functor $\ul\pi_0 \Ksg_G(k)$ satisfies induction with respect to the family $\cF_p$ of $p$-subgroups. It follows from work of Mathew--Naumann--Noel \cite[Proposition 4.4]{MNN2} that $\Ksg_G(k)$ is $\cF_p$-split. Being $\cF_p$-split is stronger than being $\cF_p$-nilpotent, which implies the stated homotopy colimit decomposition.
\end{proof}

\begin{remark}
    We highlight that \cref{cor:boundary} is where the hypothesis that $k$ is finite enters the picture. We focus on finite fields in this work, but point out the conclusion of \cref{cor:boundary}, and thus \cref{thm:main}, is true if $k$ is instead perfect and a splitting field for $G$. 

    Indeed, if $k$ is a splitting field, then the endomorphism rings appearing in \cref{prop:Gtheory} are each isomorphic to $k$. Further, it is known that the $p$-adic K-theory of a perfect field is concentrated in degree zero \cite{Hesselholt--Madsen:finitealgebras}.
\end{remark}

\begin{proof}[Proof of \cref{thm:mainColim}]
    It is a general fact that $G$-spectra which are split with respect to a family in the sense of Mathew--Naumann--Noel have such a colimit decomposition of homotopy groups \cite[Theorem 4.16]{MNN2}. The result now follows from the isomorphism $\Ksg_{n+1}(kG)\cong\K_n(kG;\bbZ_p)$ of \cref{cor:boundary}.
\end{proof}

Under the additional hypothesis of the intersection condition for the Sylow $p$-subgroup $S$, we are able to simplify the colimit appearing in 
\cref{thm:mainColim} to the coinvariants of the Weyl group action at the $S$ level. This makes use of the following handy lemma.

\begin{lemma}[{\cite[Proposition 6.31]{MNN}}] \label[lemma]{lem:cofinal}
    Let $\cF$ be a family of subgroups, and let $\cA$ be a collection of subgroups such that 
    \begin{itemize}
        \item every $H\in \cF$ is contained in an element of $\cA$, and 
        \item $\cA$ is closed under conjugation and intersection.
    \end{itemize}
    Then, the inclusion $\cO_\cA(G)\to\cO_\cF(G)$ is cofinal.
\end{lemma}

\begin{proof}[Proof of \cref{thm:mainOrb}]
    If a Sylow $p$-subgroup $S$ is trivial intersection, then by \cref{lem:cofinal}, the full subcategory of $\cO_p(G)$ spanned by the trivial subgroup and the Sylow subgroups is cofinal. Moreover, for $n\geq 1$, $p$-adic K-theory vanishes at the bottom level. Alltogether, the colimit over $\cO_p(G)$ reduces to the colimit over the subcategory spanned by the Sylow subgroups, which is the displayed coinvariants.
\end{proof}

\begin{remark}
    A more general form of the Brauer induction theorem states that for an arbitrary characteristic $p$ field $k$, the Green functor $\ul R_k$ satsfies induction with respect to the larger family of so-called \emph{hyperelementary} subgroups \cite[\S 17.2, Theorem 39]{Serre:lrfg}. These are subgroups of the form $C\rtimes Q$ where $Q$ is an $q$-group for some prime $q$ and $C$ is a cyclic $q'$-group. A similar argument as to \cref{lem:elem} shows that $\ul S_k$ also satisfies induction with respect to this family.
    
    By the above discussion, this implies that there is an induction theorem of the form 
    \[
        \K_n(kG;\bbZ_p) \cong \colim_{G/H\in\cO_{\wt{\cE}}(G)} \K_n(kH;\bbZ_p), \quad n>0
    \]
    for \emph{arbitrary} finite groups $G$, where $\wt{\cE}$ is the family of hyperelementary subgroups of $G$. We include this observation because it may be of independent interest, although we will not use it further.
\end{remark}

\section{Examples} \label{sec:ex}

This section is devoted to consequences of \cref{thm:main}.

We begin with a study of the $p$-isolation hypothesis in \cref{ssec:piso}. The term ``$p$-isolated'' is inspired by the \emph{prime graph} of Gruenberg--Kegel \cite{Williams:prime-graph}. This is the graph associated to the finite group $G$ whose vertices are prime numbers $q$ diving the order of $G$, in which two primes $q$ and $r$ are connected by an edge whenever $G$ contains an element of order $qr$. If $p$ divides the order of $G$, then $G$ is $p$-isolated exactly when $p$ is an isolated vertex of the prime graph. Williams has determined the prime graph components of finite simple groups \cite[Tables I,II]{Williams:prime-graph}. From this, the $p$-isolated finite simple groups can be read off.

In \cref{ssec:Sylp}, we show that for $p$-isolated group with order $p$ Sylow $p$-subgroups, the Weyl action appearing in \cref{thm:mainOrb} can be identified. Specifically, we prove \cref{thm:Sylp}, which leads to a number of explicit calculations of K-groups.

In the final section, \cref{ssec:bigS}, we highlight interesting examples of $p$-isolated groups with larger Sylow $p$-subgroups. For these groups, we cannot explicitly identify the Weyl action. We nonetheless record the consequences of \cref{thm:main} for these examples.

\subsection{\texorpdfstring{$p$}{p}-isolated Frobenius groups} \label[subsection]{ssec:piso}

The theory of Frobenius groups provides an interesting source of $p$-isolated groups. A \emph{Frobenius group} is a group $G$ admitting a decomposition as a semidirect product $K\rtimes H$, where $H$ acts freely on $K$. Here, we interpret freeness of the action in the pointed sense, meaning the stablizer of any nontrivial element of $K$ is trivial. In this situation, $K$ is called the \emph{Frobenius kernel} and $H$ is the \emph{Frobenius complement}.

\begin{proposition} \label[proposition]{prop:frob}
    Let $G=K\rtimes H$, where $H$ acts freely on $K$. If either
    \begin{propositionenum}
        \item \label[propositionenumi]{prop:copkernel}
        $H$ is a $p$-group and $K$ is a $p'$-group, or
        \item \label[propositionenumi]{prop:pkernel}
        $H$ is a $p'$-group and $K$ is a $p$-group, 
    \end{propositionenum}
    then $G$ is $p$-isolated and its Sylow $p$-subgroups are trivial intersection.
\end{proposition}

The proof relies on an alternate characterization of $p$-isolation in terms of centralizers, which we give without proof. Here, we say a finite group is a \emph{$p'$-group} for if its order is not divisible by $p$, and a group element is \emph{$p$-regular} is its order is not divisible by $p$.

\begin{lemma} \label[lemma]{lem:A}
    The following conditions on a finite group $G$ and prime $p$ are equivalent:
    \begin{enumerate}
        \item $G$ is $p$-isolated: there is no element of order $pq$ for $q$ a prime distinct from $p$,
        \item the centralizer of any element of order $p$ in $G$ is a $p$-group, and
        \item the centralizer of any $p$-regular element in $G$ is a $p'$-group.
    \end{enumerate}
\end{lemma}

\begin{proof}[Proof of \cref{prop:frob}]
    It is a general property of Frobenius groups $G$ that the centralizer of any nontrivial element of the Frobenius kernel $K$ is contained in $K$. We handle the two situations separately.
    \begin{enumerate}
        \item  In this situation, the $p$-regular elements are exactly those in $K$. Since $G$ is Frobenius, the centralizers of nontrivial elekents in $K$ are contained in $K$, so \cref{lem:A} gives that $G$ is $p$-isolated. The Frobenius complement $H$ is a Sylow $p$-sibgroup, and the trivial intersection property follows from general properties of Frobenius groups. Indeed, it is true for any Frobenius complement $H$ in a Frobenius group $G$ that $H\cap H^g=e$ for any $g\notin H$.
        
        \item In this situation, any element of order $p$ is contained in $K$, so its centralizer is also contained in the $p$-group $K$. By \cref{lem:A}, this proves $G$ is $p$-isolated. The subgroup $K$ is a normal Sylow $p$-subgorup of $G$, which is automatically trivial intersection. \qedhere
    \end{enumerate}
\end{proof}

\begin{example} \label[example]{ex:dih2}
    For any odd $m$, \cref{prop:copkernel} gives that the dihedral gorup $D_m=C_m\rtimes C_2$ is 2-isolated and its Sylow 2-subgroup $S=C_2$ is trivial intersection condition. In this case, the Weyl group is trivial, so by \cref{thm:mainOrb}, we have for any finite characteristic 2 field $k$,
    \[
        \K_n(kD_m;\bbZ_2) \cong \K_n(kC_2;\bbZ_2), \qquad n> 1.
    \]
\end{example}

\begin{example} \label[example]{ex:dihpn}
    For any odd prime $p$, the dihedral group $D_{p^n}=C_{p^n}\rtimes C_2$ is $p$-isolated by \cref{prop:pkernel}. In \cref{ex:dihp} below, we provide an explicit calculation of $p$-adic K-groups in the case $n=1$.
\end{example}

\begin{example} \label[example]{ex:AGL1q}
    For any prime power $q=p^r$, the group 
    \[
        \AGL_1(\bbF_q) = \bbF_q \rtimes \bbF_q^\times \cong C_p^r \rtimes C_{p^r-1}
    \]
    is $p$-isolated by \cref{prop:pkernel}. In \cref{ex:AGL1p} below, we provide an explicit calculation of $p$-adic K-groups in the case $r=1$.
\end{example}

We provide the following partial converse to \cref{prop:pkernel}.

\begin{lemma} \label[lemma]{lem:Wfree}
    If $G$ is $p$-isolated with Sylow $p$-subgroup $S$, then $W$ acts freely on $S$ by conjugation.
\end{lemma}

\begin{proof}
    Suppose $x\in N$ satisfies $xsx\nv=s$ for some $s\in S$. Since $G$ is $p$-isolated, $x$ must be contained in a $p$-subgroup of $N$, but $S$ is a normal Sylow $p$-subgroup of $N$, so $x\in S$. Thus, the element of $W$ that $x$ descends to is trivial.
\end{proof}

\begin{remark}
    \cref{lem:Wfree} implies that the normalizer $N=N_G(S)$ is a Frobenius group, with kernel $S$ and complement $W$. This restricts the possible groups that can appear as the Weyl group of a $p$-isolated group, although we will not make use of this observation. For example, any group admitting the structure of a Frobenius complement satisfies the property that all its abelian subgroups are cyclic \cite[Theorem 6.14]{Serre:finite-groups}.
\end{remark}

Recall that \cref{thm:mainOrb} states that the trivial intersection property for Sylow $p$-subgroups reduces the colimit appearing in the main theorem to coinvariants of the Weyl group action. We show this additional hypothesis is automatically satisfied if $S$ is abelian.

\begin{lemma} \label[lemma]{lem:Sabelian}
    If $G$ is $p$-isolated and a Sylow subgroup $S$ is abelian, then $S$ is a trivial intersection subgroup.
\end{lemma}

\begin{proof}
    Suppose that $S\cap gSg\nv$ is nontrivial for some $g\in G$. We show that $g$ normalizes $S$. This implies there are nontrivial $x,y\in S$ such that $x=gyg\nv$. Since $G$ is $p$-isolated and $S$ is abelian, we have $C_G(x)=C_G(y)=S$. Furthermore, we have 
    \[ 
        S = C_G(x) = C_G(gyg\nv) = gC_G(y)g\nv = gSg\nv,
    \]
    which proves the claim.
\end{proof}

\begin{example} \label[example]{ex:PSU3}
    Is it possible for the conclusion of \cref{lem:Sabelian} to hold, even with a nonabelian Sylow $p$-subgroup. As an example, consider the Frobenius group $G=C_3^2\rtimes Q_8$ arising from the unique irreducible 2-dimensional representation of $Q_8$. This group is also known as $\mathrm{PSU}_3(\bbF_2)$ or the Mathieu group $M_9$. 
    
    This representation is free (away from the origin), so by \cref{prop:copkernel}, $G$ is 2-isolated and its Sylow 2-subgroups $S\cong Q_8$ are trivial intersection, although they are nonabelian. The Weyl group of $S=Q_8$ in $G$ is trivial, so we obtain 
    \[
        \K_n(kG;\bbZ_2) \cong \K_n(kQ_8;\bbZ_2), \quad n\geq 1.
    \]
    as a consequence of \cref{thm:mainOrb} for any finite characteristic 2 field $k$. However, the K-groups of $kQ_8$ remain out of reach, so we cannot provide a more explicit calculation.
\end{example}

\subsection{Sylow subgroups of order \texorpdfstring{$p$}{p}} \label[subsection]{ssec:Sylp}

In this section, we focus on the case of $p$-isolated groups $G$ with Sylow subgroups of order $p$ and prove \cref{thm:Sylp}, which gives explicit K-theory calculations for such groups. By \cref{thm:mainOrb},  
\[
    \K_n(kG;\bbZ_p) \cong \K_n(kC_p;\bbZ_p)/W, \quad n>0.
\]
The most intricate part is determining the action of the Weyl group $W$ on the higher K-groups of $kC_p$. To understand this, we must recall some aspects of the calculation of $\K_*(kC_p)$ via trace methods. This is originally due to Madsen \cite{Madsen:survey}, but we follow the more modern approach of Hesselholt--Nikolaus \cite{Hesselholt--Nikolaus:handbook}. The main result is as follows.

\begin{proposition} \label[proposition]{prop:Wmodule}
    Under the hypotheses of \cref{thm:Sylp}, there is an isomorphism of $W$-modules 
    \[
        \K_n(kC_p;\bbZ_p) \cong \pi_n \left( C_p \otimes \sus \THH(k)_{h\bbT_p} \right), \quad n>0.
    \]
\end{proposition}

\begin{proof}
    The cyclotomic trace can be lifted to a map of $G$-spectra $\TC_G(k)\to \K_G(k)$ \cite[Example 2.19]{CMNN2}. Here, $\TC_G(k)$ is a $G$-spectrum with $H$-fixed points $\TC_G(k)^H\simeq \TC(kH)$. The relevant result for us is that in positive degrees and after $p$-completion, this is an isomorphism \cite[Theorem D]{Hesselholt--Madsen:finitealgebras}. It therefore suffices to compute $C_p$-spectrum $\TC_G(k)^{C_p}$.

    The assembly map induces a cofiber sequence of $G$-spectra
    \[
        EG_+ \otimes \TC_G(k) \to \TC_G(k) \to \WhTC_G(k) = \tEG \otimes \TC_G(k).
    \]
    This is a special case of the $\cF$-isotropy separation sequence, where $\cF$ consists of only the trivial subgroup. The Hesselholt--Nikolaus result can be interpreted as determining the nonequivariant stable homotopy type of the cofiber at the $S\cong C_p$ level. They show there is an equivalence of nonequivariant spectra 
    \[
        \WhTC_G(k)^{C_p} \simeq C_p \otimes \sus\THH_{h\bbT_p}(k)
    \]
    where $C_p$ is regarded as a pointed set \cite[Theorem 15.4.1]{Hesselholt--Nikolaus:handbook}. 

    Madsen \cite[Lemma 5.1.19]{Madsen:survey} showed that the long exact sequence on homotopy groups associated to the assembly cofiber sequence reduces to the short exact sequence
    \[
        0 \to H_{2i-1}(BC_p;k) \to \TC_{2i-1}(kC_p) \to \pi_{2i-1} \WhTC_G(k)^{C_p} \to H_{2i-1}(BC_p;k) \to 0.
    \]
    By the above discussion, this is a short exact sequence of $W$-modules. The result follows from the following lemma, which uses that and $kW$ is a semisimple ring (since $W$ has order coprime to $p$), so all modules are simultaneously projective and injective.
\end{proof}

\begin{lemma} \label[lemma]{lem:projinj}
    Given an exact sequence in an abelian category
    \[
        0 \to A \xrightarrow{f} B \xrightarrow{g} C \xrightarrow{h} A \to 0
    \]
    where $A$ is both injective and projective, there exists an abstract isomorphism $B\cong C$.
\end{lemma}

\begin{proof}
    The extension $\im(g) \to C \to A$ is split since $A$ is projective, so $C\cong A\oplus \im(g)$. Likewise, the extension $A\to B\to B/\im(f)$ is split since $A$ is injective, so $B\cong A\oplus B/\im(f)$. The result follows from the isomorphisms 
    \[
        B/\im(f) \cong B/\ker(g) \cong \im(g). \qedhere
    \]
\end{proof}

\begin{proof}[Proof of \cref{thm:Sylp}]
    The Weyl group $W\leq \Aut(C_p)$ acts on $C_p$ as a pointed set, and thus on the spectrum $C_p \otimes \sus\THH(k)_{h\bbT_p}$. To see this, the action of $W$ on 
    \[
        \THH(kG) \simeq \THH(k) \otimes LBG_+ \simeq \THH(k) \otimes \Big( \coprod_{x\in C_p} BC_G(x)_+ \Big)
    \]
    permutes the non-identity summands in the centralizer decomposition. These non-identity summands comprise the cofiber of the assembly map on $\THH$, and $\TC$ respects this decomposition into summands \cite[Proof of Theorem 15.4.1]{Hesselholt--Nikolaus:handbook}.
    
    Away from the basepoint, the $W$ action on the pointed set $C_p$ has $(p-1)/|W|$ many orbits, since the action is free by \cref{lem:Wfree}. Hesselholt--Nikolaus show \cite[Remark 15.4.8]{Hesselholt--Nikolaus:handbook}
    \[
        \pi_{2i}\THH(k)_{h\bbT_p} \cong k^i, \quad i\geq 0.
    \]
    Alltogether, we conclude that the coinvariants satisfy
    \[
        \K_{2i-1}(kC_p;\bbZ_p)/W \cong \pi_{2i-1} \left(C_p \otimes \sus\THH(k)_{h\bbT_p}\right)/W \cong k^{\frac{p-1}{|W|}i}, \quad i>0. \qedhere
    \]
\end{proof}

Wtih this formula in hand, we turn to applications.

\begin{example} \label[example]{ex:dihp}
    For any odd prime $p$, the dihedral group $D_p$ is $p$-isolated with Weyl group $W\cong C_2$. It follows from \cref{thm:Sylp} that 
    \[
        \K_{2i-1}(kD_p;\bbZ_p) \cong k^{\frac{p-1}{2}i}, \quad i>0.
    \]
\end{example}

\begin{example} \label[example]{ex:symp}
    The symmetric group $\Sigma_n$ is $p$-isolated for $n\leq p+1$. Indeed, $\Sigma_n$ has order coprime to $p$ for $n<p$, and the only elements with order dividing $p$ in $\Sigma_p$ or $\Sigma_{p+1}$ have order $p$ exactly. 
    
    In either $\Sigma_p$ and $\Sigma_{p+1}$, the Weyl group $W$ has order $p-1$. It follows from \cref{thm:Sylp} that
    \[
        \K_{2i-1}(k\Sigma_p;\bbZ_p) \cong \K_{2i-1}(k\Sigma_{p+1};\bbZ_p) \cong k^{i}, \quad i>0.
    \]
\end{example}

\begin{remark} \label[remark]{rmk:LRRV}
    As an application of their work on assembly maps for topological cyclic homology of group algebras, Lück--Reich--Rognes--Varisco show \cite[Proposition 1.3]{LRRV2}
    \[
        \TC(A\Sigma_3;\bbZ_p) \simeq \TC(AC_2;\bbZ_p) \oplus \widetilde{\TC}(AC_3;\bbZ_p)_{hC2}
    \]
    for connective ring spectra $A$. We compare this to our result that 
    \[
        \Ksg(k\Sigma_3) \simeq \Ksg(kC_3)_{hC_2}.
    \]
    for $k$ a finite characteristic 3 field.
\end{remark}

We generate many more examples from \cref{ex:symp} by passing to subgroups. 

\begin{corollary} \label[corollary]{cor:Aaction}
    If $G$ is a finite group which admits a faithful transitive action on a set of size at most $p+1$, then $G$ is $p$-isolated and its Sylow $p$-subgroups are trivial intersection.
\end{corollary}

\begin{proof}
    Any such $G$ embeds into $\Sigma_{p+1}$ by hypothesis, which is $p$-isolated by \cref{ex:symp}. The condition of $p$-isolation is evidently closed under passing to subgroups, so $G$ is $p$-isolated as well. Moreover, any Sylow $p$-subgroup of $G$ is trivial intersection since it is contained in a Sylow $p$-subgroup of $\Sigma_p$, which has order $p$.
\end{proof}

\begin{example} \label[example]{ex:altp}
    The alternating group $A_n$ is $p$-isolated for $n\leq p+2$. Indeed, if $n\leq p+1$, this is covered by \cref{cor:Aaction}. For $n=p+2$, the same observation holds, namely that the only elements of order dividing $p$ have order $p$ exactly.
    
    We assume $p$ is odd, since $A_2$ is trivial, and $A_3$ is cyclic, and $A_4$ has a larger Sylow 2-subgroup; see \cref{ex:A45-2}. In this case, the groups $A_p$ and $A_{p+1}$ have Weyl groups of order $(p-1)/2$. It follows from \cref{thm:Sylp} that
    \[
        \K_{2i-1}(kA_p;\bbZ_p) \cong \K_{2i-1}(kA_{p+1};\bbZ_p) \cong k^{2i}, \quad i>0.
    \]
    The Weyl group of $A_{p+2}$ has order $p-1$, so
    \[
        \K_{2i-1}(kA_{p+2};\bbZ_p) \cong k^{i}, \quad i>0.
    \]
\end{example}

\begin{example} \label[example]{ex:AGL1p}
    For any $p$, the group
    \[ 
        \AGL_1(\bbF_p) = \bbF_p \rtimes \bbF_p^\times \cong C_p \rtimes C_{p-1}
    \]
    is $p$-isolated, either as a special case of \cref{ex:AGL1q} or following from \cref{cor:Aaction}. The Weyl group $W\cong\Aut(C_p)$ has order $p-1$, so \cref{thm:Sylp} gives
    \[
        \K_{2i-1}(k\AGL_1(p);\bbZ_p) \cong k^i, \quad i>0.
    \]
\end{example}

\begin{remark} \label[remark]{rmk:HN}
    In their survey article, Hesselholt--Nikolaus analyze the assembly cofiber sequence for the TC of group algebras of $G=\AGL_1(\bbF_p)\cong C_p\rtimes \Aut(C_p)$. Here, we sketch how this leads to an alternate proof of the calculation in \cref{ex:AGL1p}. 
    
    For $k$ a finite characteristic $p$ field, their result \cite[Theorem 15.4.9]{Hesselholt--Nikolaus:handbook} identifies
    \[
        \WhTC_G(k)^G \simeq \sus\THH(k)_{h\bbT_p}\oplus \TC(k)\otimes B\Aut(C_p).
    \]
    The summand $\TC(k)\otimes B\Aut(C_p)$ has homotopy groups concentrated in nonpositive degrees, so the long exact sequence associated to the assembly cofiber sequence reduces to 
    \[
        0 \to H_{2i-1}(BG;k) \to \TC_{2i-1}(kG) \to \pi_{2i-1} \WhTC_G(k)^G \to H_{2i-1}(BG;k) \to 0
    \]
    for $i>0$. By \cref{lem:projinj}, this results in an isomorphism
    \[
        \K_{2i-1}(kG;\bbZ_p) \cong \TC_{2i-1}(kG) \cong \pi_{2i} \THH(k)_{h\bbT_p} \cong k^i, \qquad i>0.
    \]
\end{remark}

\begin{example} \label[example]{ex:PSL2}
    For any $p$, the group $\PSL_2(p)$ is $p$-isolated by \cref{cor:Aaction}, as it acts faithfully on the $p+1$ points of the projective line over $\bbF_p$. If $p=2$, there is an isomorphism $\PSL_2(2)\cong \Sigma_3$, so the 2-adic K-theory can be determined from \cref{ex:symp}. For $p$ odd, the Weyl group has order $(p-1)/2$. It follows from \cref{thm:Sylp} that
    \[
        \K_{2i-1}(k\PSL_2(p);\bbZ_p) \cong k^{2i}, \quad i>0.
    \]
\end{example}

We finish this section by illustrating how to use these methods to obtain complete integral calculations of K-groups. We carry this out for the example of $k\Sigma_p$, which was addressed in \cref{ex:symp}. The calculation uses knowledge of the representation theory of $\Sigma_p$ over $k$, and in particular, the splitting fields of irreducible representations. An similar calculation is possible for the above groups, given similar representation-theoretic input.

\begin{corollary} \label[corollary]{cor:FpSp}
    The algebraic K-groups of $\bbF_{p^r}\Sigma_p$ are
    \[ 
    \K_n(\bbF_p\Sigma_p) \cong \begin{cases}
        \bbZ^c & n = 0, \\
        (\bbZ/p^{ri}-1)^c \oplus (\bbZ/p)^{ri} & n = 2i-1>0, \\
        0 & \textrm{otherwise,}
    \end{cases} \]
    where $c$ denotes the number of irreduble representations of $\Sigma_p$ over $\bbF_p$.
\end{corollary}

\begin{proof}
    It is known that any field $k$ is a splitting field for the symmetric groups. It follows that the endomorphism rings appearing in the G-theory formula \cref{prop:Gtheory} are all the ground field $k$ \cite[Proposition 9.2.5]{Webb:book}. By the Dress--Kuku Theorem \cite{Dress--Kuku}, this gives the calculation away from the prime $p$. The $p$-local summand is given by \cref{ex:symp}.
\end{proof}

\subsection{Larger Sylow subgroups} \label{ssec:bigS}

If the Sylow $p$-subgroups have order larger than $p$, then we are unable to explicitly identify the Weyl action on $\K_*(kS)$ as in the previous section. Nonetheless, we record some interesting examples of $p$-isolated groups with larger Sylow $p$-subgroups and state consequences of the main theorem for these.

\begin{example} \label[example]{ex:A45-2}
    At the prime $p=2$, $A_4$ and $A_5$ are both 2-isolated with trivial intersection Sylow 2-subgroups $S\cong C_2\times C_2$ by \cref{lem:Sabelian}. The Sylow subgroup is normal in $A_4$ with Weyl group $W\cong C_3$, and $A_4$ is the normalizer of the Sylow 2-subgroup in $A_5$. It follows from \cref{thm:main} that for $k$ any finite characteristic 2 field,
    \[
        \K_n(kA_4;\bbZ_2) \cong \K_n(kA_5;\bbZ_2) \cong \K_n(kC_2^2;\bbZ_2)/C_3, \quad n\geq 1.
    \]    
    Forthcoming work of Moselle computes the K-groups $\K_*(kC_p^n)$. However, at present, we are unable to identify the Weyl actions explicitly.
\end{example}

\begin{example} \label[example]{ex:A6-3}
    At the prime $p=3$, a similar story plays out for $G=A_6$ as in previous example. The Sylow 3-subgroup is $S\cong C_3\times C_3$ which is trivial intersection by \cref{lem:Sabelian}, and the Weyl group is $W\cong C_4$. It follows from \cref{thm:main} that for $k$ any finite characteristic 3 field,
    \[
        \K_n(kA_6;\bbZ_3) \cong \K_n(kC_3^2;\bbZ_3)/C_4, \quad n\geq 1.
    \]
\end{example}

We conclude with some exotic examples at the prime $p=2$.

\begin{example} \label[example]{ex:S4-2}
    The group $\Sigma_4$ is 2-isolated with self-normalizing Sylow 2-subgroup $S\cong D_4$, but $S$ is not trivial intersection. Indeed, the Sylow 2-subgroups of $\Sigma_4$ all contain a normal Klein four subgroup generated by the conjugacy class of $(1\,2)(3\,4)\in\Sigma_4$. Therefore, \cref{thm:mainColim} does not apply, but we may still use the more genreal \cref{thm:mainOrb}. By \cref{lem:cofinal}, we may reduce the colimit over $\cO_2(\Sigma_4)$ to the subcategory spanned by the Sylow 2-subgroups and this normal Klein four subgroup. We conclude that for a finite characteristic 2 field $k$, 
    \[
        \K_n(k\Sigma_4;\bbZ_2) \cong \colim\left( 
            \begin{tikzcd}
                \K_n(kC_2^2;\bbZ_2) \arrow[loop, looseness = 2, "\Sigma_4/C_2^2\cong\Sigma_3"'] \arrow[r] \arrow[r, shift left = 2] \arrow[r, shift right = 2] & 
                \K_n(kD_4;\bbZ_2)
            \end{tikzcd}
        \right), \quad n\geq 1.
    \]
\end{example}

\begin{example} \label[example]{ex:A6-PSL27-2}
    A similar story as for $\Sigma_4$ plays out for the groups $A_6$ and $\PSL_2(7)$: they are 2-isolated with self-normalizing Sylow 2-subgroup $S\cong D_4$ which is also not trivial intersection. We leave it to the interested reader to formulate an analysis similar to \cref{ex:S4-2}.
\end{example}

\noindent \textsc{Department of Mathematics, Cornell University, Ithaca, NY 14853} \\
\noindent \emph{Email address:} \texttt{cpv29@cornell.edu}

\end{document}